\documentclass{article}
\usepackage{geometry}
\usepackage[british]{babel}
\usepackage{ae}
\usepackage{latexsym}
\usepackage{amsfonts}
\usepackage{amssymb}
\usepackage{amsthm}
\usepackage{graphicx}
\usepackage{multirow}
\usepackage{color}
\usepackage{youngtab}
\usepackage{young}
\usepackage{amsmath}
\usepackage{amscd}
\usepackage{float}
\usepackage{hyperref}

\newcounter{itemcounter}
\numberwithin{itemcounter}{subsection}

\newtheorem{thm}[itemcounter]{Theorem}

\newtheorem*{thm*}{Theorem}
\newtheorem*{con*}{Conjecture}
\newtheorem*{cor*}{Corollary}
\newtheorem*{ack*}{Acknowledgements}

\title{A note on perfect isometries between finite general linear and unitary groups at unitary primes}
\author{Michael Livesey}
\date{}

\begin{document}

\maketitle

\begin{abstract}
Let $q$ be a power of a prime, $l$ a prime not dividing $q$, $d$ a positive integer coprime to both $l$ and the multiplicative order of $q\mod l$ and $n$ a positive integer.
A. Watanabe proved that there is a perfect isometry between the principal $l-$blocks of $\operatorname{GL}_n(q)$ and $\operatorname{GL}_n(q^d)$ where the
correspondence of characters is give by Shintani descent. In the same paper Watanabe also prove that if $l$ and $q$ are odd and $l$ does not divide
$|\operatorname{GL}_n(q^2)|/|\operatorname{U}_n(q)|$ then there is a perfect isometry between the principal $l-$blocks of $\operatorname{U}_n(q)$ and $\operatorname{GL}_n(q^2)$
with the correspondence of characters also given by Shintani descent. R. Kessar extended this first result to all unipotent blocks of $\operatorname{GL}_n(q)$ and
$\operatorname{GL}_n(q^d)$. In this paper we extend this second result to all unipotent blocks of $\operatorname{U}_n(q)$ and $\operatorname{GL}_n(q^2)$. In particular this proves
that any two unipotent blocks of $\operatorname{U}_n(q)$ at unitary primes (for possibly different $n$) with the same weight are perfectly isometric. We also prove that this perfect
isometry commutes with Deligne-Lusztig induction at the level of characters.
\end{abstract}

\section{Introduction}\label{sec:intro}
\subsection{Perfect isometries}
Let $l$ be a prime and consider an $l-$modular system $(K,\mathcal{O},k)$, such that $K$ contains enough roots of unity for the groups being considered in this paper. Let $G$ and $H$
be finite groups and $b$ and $c$ block idempotents of $\mathcal{O}G$ and $\mathcal{O}H$ respectively. We denote by $\operatorname{Irr}(G,b)$ the set of irreducible characters of
$\mathcal{O}Gb$. A perfect isometry (see~\cite{broue1990}) between $\mathcal{O}Gb$ and $\mathcal{O}Hc$ is an isometry
\begin{align*}
&I:\mathbb{Z}\operatorname{Irr}(G,b)\to\mathbb{Z}\operatorname{Irr}(H,c),\\
\text{with }&I(\mathbb{Z}\operatorname{Irr}(G,b))=\mathbb{Z}\operatorname{Irr}(H,c),
\end{align*}
where
\begin{align*}
\hat{I}:G\times H&\to\mathcal{O}\\
(x,y)&\mapsto\sum_{\chi\in\operatorname{Irr}(G,b)}\chi(x)I(\chi)(y)
\end{align*}
has the property that $\hat{I}(x,y)$ is divisible by $|C_G(x)|$ and $|C_G(y)|$ in $\mathcal{O}$ for all $x\in G$ and $y\in H$ and that $\hat{I}(x,y)=0$
if in addition exactly one of $x$ and $y$ is $l-$singular.

\subsection{Shintani descent}

Let $n$ be a positive integer and $q$ a power of a prime $p$. We denote by $F$ be the Frobenius endomorphism of $\mathbf{G}:=\operatorname{GL}_n(\overline{\mathbb{F}_q})$ defined by
\begin{align*}
F:\mathbf{G}&\to\mathbf{G}\\
A&\mapsto JA^{-t[q]}J,
\end{align*}
where $A^{[q]}$ is the matrix obtained from $A$ by raising every entry to the power $q$, $A^{-t}$ is the inversed transpose of $A$ and $J$ is the $n\times n$ matrix with $1$'s on the
anti-diagonal and $0$'s everywhere else. We define
\begin{align*}
\mathbf{G}^F:=&\operatorname{U}_n(q)\\
\mathbf{G}^{F^2}:=&\operatorname{GL}_n(q^2).
\end{align*}
We denote by $\sigma$ the restriction of $F$ to $\mathbf{G}^{F^2}$. Two elements $g,h\in\mathbf{G}^{F^2}$ are said to be $\sigma-$conjugate if $xg\sigma(x)^{-1}=h$ for some
$x\in\mathbf{G}^{F^2}$. If $g\in\mathbf{G}^{F^2}$ then $g\sigma(g)$ is conjugate in $\mathbf{G}^{F^2}$ to an element of $\mathbf{G}^F$. This induces a bijection $N_{F^2/F}$ between
the $\sigma-$conjugacy classes in $\mathbf{G}^{F^2}$ and the conjugacy classes in $\mathbf{G}^F$.
\newline
\newline
If $G$ is a finite group and $X\subset G$ is a union of conjugacy classes of $G$ then we denote by $\mathcal{C}(X)$ the set of functions $X\to\mathcal{O}$ invariant on conjugacy
classes. If $g,h\in\mathbf{G}^{F^2}$ then $g$ and $h$ are $\sigma-$conjugate if and only if $g\sigma$ and $h\sigma$ are conjugate in $\mathbf{G}^{F^2}\rtimes\langle\sigma\rangle$.
With this in mind we define
\begin{align*}
\operatorname{Sh}_{F^2/F}:\mathcal{C}(\mathbf{G}^{F^2}\sigma)\to\mathcal{C}(\mathbf{G}^F)
\end{align*}
to be the map induced by $N_{F^2/F}$. The following theorem is due to Kawanaka~\cite[Theorem 4.1]{kawanaka1977}.

\begin{thm}\label{thm:shintani}
Let $\chi$ be an irreducible character of $\mathbf{G}^{F^2}$ stable under $\sigma$. Then there is a unique extension $\tilde{\chi}$ of $\chi$ to
$\mathbf{G}^{F^2}\rtimes\langle\sigma\rangle$ such that $\operatorname{Sh}_{F^2/F}(\tilde{\chi})$ is an irreducible character of $\mathbf{G}^F$. Moreover,
$\chi\mapsto\operatorname{Sh}_{F^2/F}(\tilde{\chi})$ gives a bijection between the set of irreducible characters of $\mathbf{G}^{F^2}$ stable under $\sigma$ and the
set of irreducible characters of $\mathbf{G}^F$.
\end{thm}

We remark that a corresponding theorem was proved for the characters of $\operatorname{GL}_n(q)$ and $\operatorname{GL}_n(q^d)$, for some integer $d$, in theorems of
T. Shintani~\cite[Theorem 1, Lemma 2-11]{shintani1976} and C. Bonnaf\'{e}~\cite[Theorem 4.3.1]{bonnaf1998}. We therefore refer to the bijection in
Theorem~\ref{thm:shintani} as Shintani descent.

\section{The characters and unipotent blocks of \texorpdfstring{$\operatorname{GL}_n(q^2)$}{TEXT} and \texorpdfstring{$\operatorname{U}_n(q)$}{TEXT}}

We follow P. Fong and B. Srinivasan~\cite{fonsri1982} for the parametrisation of the characters and blocks of $\operatorname{GL}_n(q^2)$ and $\operatorname{U}_n(q)$.
If $\Gamma(X)\in\mathbb{F}_{q^2}[X]$ we denote by $d_\Gamma$
the degree of $\Gamma$. We set $\Delta$ to be the set of monic irreducible polynomials over $\mathbb{F}_{q^2}$ with non-zero roots. We then
define the following auotmorphism on $\Delta$:
\begin{align*}
\Gamma=&(X^m+a_{m-1}X^{m-1}+\dots+a_1X+a_0)\mapsto\\
\widetilde{\Gamma}=&{a_0}^{-q}({a_0}^{q}X^m+{a_1}^{q}X^{m-1}+\dots+{a_{m-1}}^{q}X+1),
\end{align*}
and define
\begin{align*}
\Lambda_1:=&\{\Gamma\mid \Gamma\in\Delta,\widetilde{\Gamma}=\Gamma\},\\
\Lambda_2:=&\{\Gamma\widetilde{\Gamma}\mid \Gamma\in\Delta,\widetilde{\Gamma}\neq\Gamma\}\\
\Lambda:=&\Lambda_1\cup\Lambda_2.
\end{align*}
The irreducible characters of $G=\operatorname{GL}_n(q^2)$ (respectively $G=\operatorname{U}_n(q)$) are labelled by pairs $(s,\lambda)$, where $s$ is a semisimple element of
$\operatorname{GL}_n(q^2)$ (respectively $\operatorname{U}_n(q)$) and $\lambda=(\lambda_\Gamma)_\Gamma$ is a multipartition with $\Gamma$ running over
$\Delta$ (respectively $\Lambda$) and $|\lambda_\Gamma|=n_\Gamma(s)$, where $n_\Gamma(s)$ is the multiplicity of $\Gamma$ in the characteristic polynomial of $s$.
(In fact the conjugacy class of $s$ is completely determined by the $n_\Gamma(s)$'s.)
We denote such a character by $\chi^{GL}_{s,\lambda}$ (respectively $\chi^{U}_{s,\lambda}$) and $\chi^G_{s,\lambda}=\chi^G_{t,\mu}$ if and only if $s\sim_Gt$
and $\lambda_\Gamma=\mu_\Gamma$ for all $\Gamma$. In the case of $s=1$ we call the character unipotent.
\newline
\newline
Set $e$ to be the multiplicative order of $q^2\mod l$ (respectively $-q\mod l$). For $G=\operatorname{U}_n(q)$ we have the notion of $l$ being a linear or unitary prime
depending on whether $e$ is even or odd respectively. For $\Gamma\in\Delta\cup\Lambda_2$ we define
\begin{align*}
e_\Gamma:=
\begin{cases}
\text{the additive order of }2d_\Gamma\mod e&\text{if }G=\operatorname{U}_n(q),\Gamma\in\Lambda_1\text{ and } l \text{ is linear}\\
\text{the additive order of }d_\Gamma\mod e&\text{otherwise}.
\end{cases}
\end{align*}
The unipotent $l-$blocks of $G=\operatorname{GL}_n(q^2)$ (respectively $G=\operatorname{U}_n(q)$) are labelled by $e-$cores $\lambda$ of partitions of $n$. We label such a block
$B^{GL}_{1,\lambda}$ (respectively $B^{U}_{1,\lambda}$). Then any sylow $l-$subgroup $D$ of the subgroup $\operatorname{GL}_{n-|\lambda|}(q^2)\leq\operatorname{GL}_n(q^2)$
(respectively $\operatorname{U}_{n-|\lambda|}(q)\leq\operatorname{U}_n(q)$) is a defect group for $B^G_{1,\lambda}$. An irreducible character $\chi^G_{t,\mu}$ lies in this block
if and only if $t\sim_Gx$ for some $x\in D$, $\mu_{X-1}$ has $e-$core $\lambda$ and $\lambda_\Gamma$ has empty $e_\Gamma-$core for all other $\Gamma$.
\newline
\newline
For our main theorem we will need to know the degrees of the unipotent characters of $G$. For our description
we again follow~\cite[$\S1$]{fonsri1982}.
\newline
\newline
Let $\lambda=(\lambda_1,\lambda_2,\dots)\vdash n$ and $\lambda'=(\lambda_1',\lambda_2',\dots)$ be its conjugate partition. This notation will continue
throughout the paper. We introduce the polynomial
\begin{align*}
P_\lambda(X):=\frac{X^d(X^n-1)(X^{n-1}-1)\dots(X-1)}{\prod_h(X^h-1)},
\end{align*}
where $d:=\sum_i\lambda_i'^2-\sum_ii\lambda_i$ and $h$ runs over the hook lengths of $\lambda$. Then $\chi^{GL}_{1,\lambda}(1)=P_\lambda(q^2)$ and
$\chi^{U}_{1,\lambda}(1)=|P_\lambda(-q)|$.

\section{Main theorem}

%
%
For any multipartition $\lambda$ indexed by $\Delta$ we denote by ${}^F\lambda$ the multipartition given by ${}^F\lambda_{\widetilde{\Gamma}}=\lambda_\Gamma$.
We now state our main theorem.

\begin{thm}\label{thm:perfisom}
Suppose $l$ and $q$ are odd with $l$ a unitary prime with respect to $q$ and let $\lambda$ be an $e-$core of a partition of $n$. Then all irreducible
characters of $B^{GL}_{1,\lambda}$ are $\sigma$-stable and $B^{GL}_{1,\lambda}$ and $B^{U}_{1,\lambda}$ are perfectly isometric with the bijection of
unipotent characters given by Shintani descent.
\end{thm}

\begin{proof}
From the explicit construction of the irreducible characters of $\operatorname{GL}_n(q^2)$ (see for example~\cite[$\S12$]{digmic1991}) we see that the
permutation on the characters induced by $F:\operatorname{GL}_n(q^2)\to\operatorname{GL}_n(q^2)$ sends $\chi^{GL}_{s,\lambda}$ to
$\chi^{GL}_{{}^Fs,{}^F\lambda}$.
\newline
\newline
Suppose $\Gamma\in\Delta$ has some of the primitive
${l^a}^{\operatorname{th}}$ roots of unity as its roots for some $a\geq0$. So $d_\Gamma$ is the minimal integer satisfying $l^a\mid(q^{2d_\Gamma}-1)$.
Now $l\nmid((-q)^{d_\Gamma}+1)$ as otherwise $-q$ would have even order $\mod l$ and so $l^a\mid((-q)^{d_\Gamma}-1)$
Therefore $\omega^{(-q)^{d_\Gamma}}=\omega$ for any root $\omega$ of $\Gamma$ and
\begin{align*}
\{\omega,\omega^{q^2},\omega^{q^4}\dots\}=\{\omega,\omega^{-q},\omega^{(-q)^2}\dots\}.
\end{align*}
So the roots of $\Gamma$ are left invariant under $(*)^{-q}$ and hence $\widetilde{\Gamma}=\Gamma$. This together with the first
paragraph implies that all irreducible characters of $B^{GL}_{1,\lambda}$ are $\sigma$-stable.
\newline
\newline
We now proceed as in the proofs of~\cite[Theorem 3]{Watanabe1999} and~\cite[Theorem 1.2]{Kessar2004}. For each $\chi\in B^{GL}_{1,\lambda}$ we denote by
$\tilde{\chi}$ the specified extension to $\operatorname{GL}_n(q^2)\rtimes\langle\sigma\rangle$ in Theorem~\ref{thm:shintani}. Let $(\sigma)$ be the sum
of the elements in the conjugacy class of $\operatorname{GL}_n(q^2)\rtimes\langle\sigma\rangle$ containing $\sigma$. Then
\begin{align*}
w_{\tilde{\chi}}((\sigma))=\frac{|\operatorname{GL}_n(q^2)|}{|\operatorname{U}_n(q)|}\frac{\operatorname{Sh}_{F^2/F}(\tilde{\chi})(1)}{\chi(1)},
\end{align*}
where $w_{\tilde{\chi}}$ denotes the central character of $\operatorname{GL}_n(q^2)\rtimes\langle\sigma\rangle$ corresponding to $\tilde{\chi}$. Note that
$|\operatorname{GL}_n(q^2)|_l=|\operatorname{U}_n(q)|_l$ and so the proof
proceeds exactly as in~\cite{Watanabe1999} and~\cite{Kessar2004} once we have shown
\begin{align}\label{note:deg}
\frac{\operatorname{Sh}_{F^2/F}(\tilde{\chi})(1)}{\chi(1)}\not\equiv0\mod l\text{ for all unipotent }
\chi\in\operatorname{Irr}(\operatorname{GL}_n(q^2)).
\end{align}
These two proofs also show that once the above statement is proved for one character it is immediately proved for all characters in the same block.
Furthermore, the image of this block under $\operatorname{Sh}_{F^2/F}$ is a block of $\operatorname{U}_n(q)$. Both these facts will be important in the
remainder of this proof.
\newline
\newline
We prove (\ref{note:deg}) by showing that $\operatorname{Sh}_{F^2/F}(\tilde{\chi}^{GL}_{1,\lambda})=\chi^U_{1,\lambda}$ for all $\lambda\vdash n$ and
noting that $P(q^2)_l=|P(-q)|_l$. We proceed by induction on $n$. First we note that the for $n=1$ the result is clearly true as the only unipotent
character $\chi^G_{1,(1)}$ of either group is the trivial character. Next let $n\geq2$ and set
\begin{align*}
\mathbf{L}:=\operatorname{GL}_1(\overline{\mathbb{F}_q})\times\operatorname{GL}_{n-2}(\overline{\mathbb{F}_q})\times
\operatorname{GL}_1(\overline{\mathbb{F}_q})\leq\mathbf{G}:=\operatorname{GL}_n(\overline{\mathbb{F}_q}),
\end{align*}
and $\mathbf{P}$ to be an $F-$stable parabolic with Levi decomposition $\mathbf{P}=\mathbf{LU}$ for some $\mathbf{U}$.
Assume $\operatorname{Sh}_{F^2/F}(\tilde{\chi}^{GL}_{1,\lambda})=\chi^U_{1,\lambda}$ for all $\lambda$ with $|\lambda|\leq n-2$. Then by a generalised
version of~\cite[Lemma 2.8]{shintani1976} (see~\cite[Proposition 1.1]{digne1999})
\begin{align*}
\operatorname{Ind}^{\mathbf{G}^{F^2}\rtimes\langle\sigma\rangle}_{\mathbf{P}^{F^2}\rtimes\langle\sigma\rangle}
\operatorname{Inf}^{\mathbf{P}^{F^2}\rtimes\langle\sigma\rangle}_{\mathbf{L}^{F^2}\rtimes\langle\sigma\rangle}(\tilde{\chi})(\sigma x)=
\operatorname{Ind}^{\mathbf{G}^{F}}_{\mathbf{P}^{F}}\operatorname{Inf}^{\mathbf{P}^{F}}_{\mathbf{L}^{F}}
(\operatorname{Sh}_{F^2/F}(\tilde{\chi}))(N_{F^2/F}(x))\\
\text{for all }\sigma-\text{stable }\chi\in\mathbf{L}^{F^2}\cong\operatorname{GL}_1(q^2)\times\operatorname{GL}_{n-2}(q^2)\times\operatorname{GL}_1(q^2)
\text{ and }x\in\mathbf{G}^{F^2}.
\end{align*}
Now assume $\operatorname{Sh}_{F^2/F}(\tilde{\chi}^{GL}_{1,\lambda})=\chi^U_{1,\lambda}$ for all $|\lambda|<n$. Set
\begin{align*}
\chi=\chi^{GL}_{1,(1)}\otimes\chi^{GL}_{1,\lambda}\otimes\chi^{GL}_{1,(1)}\in\operatorname{Irr(\mathbf{L}^{F^2})}
\end{align*}
for some $\lambda\vdash(n-2)$. So by the inductive hypothesis
\begin{align*}
\operatorname{Sh}_{F^2/F}(\tilde{\chi})=
\chi^{GL}_{1,(1)}\otimes\chi^U_{1,\lambda}\in\operatorname{Irr(\mathbf{L}^{F^2})}.
\end{align*}
Then (see~\cite[Proposition 1C]{fonsri1982})
\begin{align*}
\operatorname{Ind}^{\mathbf{G}^{F^2}}_{\mathbf{P}^{F^2}}
\operatorname{Inf}^{\mathbf{P}^{F^2}}_{\mathbf{L}^{F^2}}(\chi)=2\sum_{\mu\in I_\lambda(1,1)}\chi^{GL}_{1,\mu}+\sum_{\mu\in I_\lambda(2)}\chi^{GL}_{1,\mu},
\end{align*}
where $I_\lambda(1,1)$ is the set of partitions of $n$ obtained from $\lambda$ by adding two $1-$hooks on different rows and columns and
$I_\lambda(2)$ is the set of partitions of $n$ obtained from $\lambda$ by adding a $2-$hook. So
\begin{gather*}
\operatorname{Ind}^{\mathbf{G}^{F^2}\rtimes\langle\sigma\rangle}_{\mathbf{P}^{F^2}\rtimes\langle\sigma\rangle}
\operatorname{Inf}^{\mathbf{P}^{F^2}\rtimes\langle\sigma\rangle}_{\mathbf{L}^{F^2}\rtimes\langle\sigma\rangle}(\tilde{\chi}),\\
\text{is some extension to }\mathbf{G}^{F^2}\rtimes\langle\sigma\rangle\text{ of }
2\sum_{\mu\in I_\lambda(1,1)}\chi^{GL}_{1,\mu}+\sum_{\mu\in I_\lambda(2)}\chi^{GL}_{1,\mu}.
\end{gather*}
Also (by~\cite[Proposition 1C]{fonsri1982})
\begin{align*}
\operatorname{Ind}^{\mathbf{G}^F}_{\mathbf{P}^F}\operatorname{Inf}^{\mathbf{P}^F}_{\mathbf{L}^F}
(\operatorname{Sh}_{F^2/F}(\tilde{\chi}))=\sum_{\mu\in I_\lambda(2)}\chi^U_{1,\mu},
\end{align*}
where the sum is over partitions of $n$ obtained from $\lambda$ by adding a $2-$hook. Therefore by Theorem~\ref{thm:shintani} the sets
\begin{align*}
\{\chi^{GL}_{1,\mu}\}_{\mu\in I_\lambda(2)}\text{ and }\{\chi^U_{1,\mu}\}_{\mu\in I_\lambda(2)}
\end{align*}
correspond under $\operatorname{Sh}_{F^2/F}$. If $\mu$ has $2-$weight greater than $1$ then by analysising exactly for which $\lambda$ we have
$\mu\in I_\lambda(2)$ we obtain that $\operatorname{Sh}_{F^2/F}(\tilde{\chi}^{GL}_{1,\mu})=\chi^U_{1,\mu}$. The only difficulty therfore, is
in showing the inductive step when $\mu$ has $2-$weight $0$ or $1$.
\newline
\newline
Suppose $\mu$ has $2-$weight $0$. In other words $\mu$ is a $2-$core,
say $\mu=(t,t-1,\dots,1)\vdash n$. Note that we need only show $\operatorname{Sh}_{F^2/F}(\tilde{\chi}^{GL}_{1,\mu})$ is unipotent as by induction every other
unipotent character of $\operatorname{GL}_n(q^2)$ is sent to the unipotent character with the same label under Shintani descent. By varying $l$ we
can make $e$ arbitrarily large so let's assume $e>n$. Next set
\begin{align*}
\mu_e^+:=&(t+e-1,t-1,\dots,1)\text{ and}\\ \mathbf{L}:=&\operatorname{GL}_1(\overline{\mathbb{F}_q})^{(e-1)/2}\times
\operatorname{GL}_n(\overline{\mathbb{F}_q})\times\operatorname{GL}_1(\overline{\mathbb{F}_q})^{(e-1)/2}.
\end{align*}
Then as $\chi^{GL}_{1,\mu_e^+}$ appears with multiplicity $1$ in
\begin{align*}
\operatorname{Ind}^{\mathbf{G}^{F^2}}_{\mathbf{P}^{F^2}}\operatorname{Inf}^{\mathbf{P}^{F^2}}_{\mathbf{L}^{F^2}}
(\chi_{1,(1)}^{GL\otimes(e-1)/2}\otimes\chi^{GL}_{1,\mu}\otimes\chi_{1,(1)}^{GL\otimes(e-1)/2}),
\end{align*}
$\operatorname{Sh}_{F^2/F}(\tilde{\chi^{GL}_{1,\mu_e^+}})$ must be an irreducible constituent of 
\begin{align}\label{note:constituent}
\operatorname{Ind}^{\mathbf{G}^F}_{\mathbf{P}^F}\operatorname{Inf}^{\mathbf{P}^F}_{\mathbf{L}^F}
(\chi_{1,(1)}^{GL\otimes(e-1)/2}\otimes\operatorname{Sh}_{F^2/F}(\tilde{\chi}^{GL}_{1,\mu})).
\end{align}
Therefore, we need only show $\operatorname{Sh}_{F^2/F}(\tilde{\chi}^{GL}_{1,\mu_e^+})$ is unipotent. Now certainly the unipotent character labelled by
$(t-1,t-1,\dots,1^{e+1})$ gets sent to the unipotent character with the same label under Shintani descent as we can apply the original inductive argument as above.
Therefore, as blocks go to blocks, $\operatorname{Sh}_{F^2/F}(\tilde{\chi^{GL}_{1,\mu_e^+}})$ is in the unipotent block labelled by
$(t-1,t-1,\dots,1)$. This together with (\ref{note:constituent}) gives that $\operatorname{Sh}_{F^2/F}(\tilde{\chi}^{GL}_{1,\mu_e^+})$ is unipotent.
\newline
\newline
We next consider the case of $2-$weight $1$. Similarly to above we assume $e>n$ and set
\begin{align*}
\lambda&:=(t,t-1,\dots,1)\vdash n,\\
\lambda^+&:=(t+2,t-1,\dots,1)\vdash (n+2),\lambda^-:=(t,t-1,\dots,1^3)\vdash (n+2),\\
\lambda_e^+&:=(t+e-1,t-1,\dots,1),\lambda_e^-:=(t,t-1,\dots,1^e).
\end{align*}
By the original inductive hypothesis the sets
\begin{align*}
\{\chi^{GL}_{1,\mu}\}_{\mu\in I_\lambda(2)}\text{ and }\{\chi^U_{1,\mu}\}_{\mu\in I_\lambda(2)}
\end{align*}
correspond under Shintani descent. Note that these sets both consist of two unipotent characters labelled by $\lambda^+$ and $\lambda^-$. Now if
$\operatorname{Sh}_{F^2/F}(\tilde{\chi}^{GL}_{1,\lambda^+})=\chi^U_{1,\lambda^-}$ then by the inductive argument we can show that
$\operatorname{Sh}_{F^2/F}(\tilde{\chi}^{GL}_{1,\lambda_e^+})=\chi^U_{1,\lambda_e^-}$. But as in the previous argument
$\operatorname{Sh}_{F^2/F}(\tilde{\chi^{GL}_{1,\mu_e^+}})$ is in the unipotent block labelled by $(t-1,t-1,\dots,1)$. This is a contradiction and hence the result is proved.
\end{proof}

\section{Deligne-Lusztig Induction}

By an abuse of notation we denote by $I$ simultaneously all the the perfect isometries in Theorem~\ref{thm:perfisom}. For each $\lambda\vdash n$ we denote by
$\delta_\lambda$ the sign such that $I(\chi^{GL}_{1,\lambda})=\delta_\lambda\chi^U_{1,\lambda}$ in $I$. Note that $\delta_\lambda$ is only determined
reletive to $\delta_\mu$ where $\lambda$ and $\mu$ have the same $e-$core.

\begin{thm}
The perfect isometries $I$ can be chosen such that, on unipotent characters, the following diagram commutes
\begin{align*}
\begin{CD}
\operatorname{GL}_m(q)\times\operatorname{GL}_n(q) @> R^{\mathbf{G}^{F^2}}_{\mathbf{L}^{F^2}} >> \operatorname{GL}_{m+n}(q)\\
@VV I V @VV I V\\
\operatorname{U}_m(q)\times\operatorname{U}_n(q) @> R^{\mathbf{G}^F}_{\mathbf{L}^F} >> \operatorname{U}_{m+n}(q)
\end{CD},
\end{align*}
where
\begin{align*}
\mathbf{L}:=\operatorname{GL}_m(\overline{\mathbb{F}_q})\times\operatorname{GL}_n(\overline{\mathbb{F}_q}).
\end{align*}
\end{thm}

\begin{proof}
As described in Fong and Srinivasan there exist signs $\epsilon_\lambda$ such that if we replace $I$ in the above diagram by
\begin{align*}
J:\chi^{GL}_{1,\lambda}\mapsto\epsilon_\lambda\chi^U_{1,\lambda}
\end{align*}
then the diagram is commutative and the sign of $\epsilon_\lambda$ is determined by the sign of $P_\lambda(-q)$. Therefore, we need only show that
\begin{gather*}
\frac{P_\lambda(q^2)}{|P_\lambda(-q)|}\frac{|P_\mu(-q)|}{P_\mu(q^2)}\equiv\frac{P_\lambda(-q)}{|P_\lambda(-q)|}\frac{|P_\mu(-q)|}{P_\mu(-q)}\mod l\\
\text{or equivalently }\frac{P_\lambda(q^2)}{P_\lambda(-q)}\mod l\text{ is fixed for }\lambda\text{ with a particular }e-\text{core}.
\end{gather*}
We prove this by induction on $|\lambda|$. Specfically we prove that if $\mu$ is obtained from $\lambda$ by adding an $e-$hook then
\begin{align}\label{align:equiv}
\frac{P_\lambda(q^2)}{P_\lambda(-q)}\equiv\frac{P_\mu(q^2)}{P_\mu(-q)}\mod l.
\end{align}
We first prove this for straight $e-$hooks. Let $\lambda:=(\lambda_1,\lambda_2,\dots)$. Suppose we add a straight $e-$hook to $\lambda$ to obtain
$\mu$. Let the top left corner of this hook be on the
$r^{\operatorname{th}}$ row and the $c^{\operatorname{th}}$ column and let $a-1$ be its armlength. So
\begin{align*}
\mu_r=&\lambda_r+a,\\
\mu_i=&\lambda_i+1\text{ for }(r+1\leq i\leq r+e-a),\\
\mu_i=&\lambda_i\text{ for }(i<r),(i>r+e-a).
\end{align*}
 Now
\begin{align*}
&(\sum_i\mu_i'^2-\sum_ii\mu_i)-(\sum_i\lambda_i'^2-\sum_ii\lambda_i)\\
=&(r+e-a)^2-(r-1)^2+(a-1)(r^2-(r-1)^2)-(ra+\sum_{i=r+1}^{r+e-a}i)\\
=&(r-a)e+\frac{e(e-1)}{2}+\frac{a(a-1)}{2}
\end{align*}
The hook lengths of $\mu$ are gotten from those of $\lambda$ by replacing
\begin{align*}
\begin{array}{cccc}
\lambda_1-c+r-1&\lambda_2-c+r-2&\dots&\lambda_{r-1}-c+1\\
\lambda_1-c+r-2&\lambda_2-c+r-3&\dots&\lambda_{r-1}-c\\
\vdots&\vdots&\ddots&\vdots\\
\lambda_1-c+r-a&\lambda_2-c+r-a-1&\dots&\lambda_{r-1}-c+2-a
\end{array}
\end{align*}
with
\begin{align*}
\begin{array}{cccc}
\lambda_1-c+r+e-a&\lambda_2-c+r+e-a-1&\dots&\lambda_{r-1}-c+e-a+2\\
\lambda_1-c+r-1&\lambda_2-c+r-2&\dots&\lambda_{r-1}-c+1\\
\vdots&\vdots&\ddots&\vdots\\
\lambda_1-c+r-a+1&\lambda_2-c+r-a&\dots&\lambda_{r-1}-c+3-a
\end{array},
\end{align*}
by replacing
\begin{align*}
\begin{array}{cccc}
\lambda_1'-r+c-1&\lambda_2'-r+c-2&\dots&\lambda_{c-1}'-r+1\\
\lambda_1'-r+c-2&\lambda_2'-r+c-3&\dots&\lambda_{c-1}'-r\\
\vdots&\vdots&\ddots&\vdots\\
\lambda_1'-r+c-e+a-1&\lambda_2'-r+c-e+a-2&\dots&\lambda_{c-1}'-r-e+a+1
\end{array}
\end{align*}
with
\begin{align*}
\begin{array}{cccc}
\lambda_1'-r+c-1+a&\lambda_2'-r+c-2+a&\dots&\lambda_{c-1}'-r+1+a\\
\lambda_1'-r+c-1&\lambda_2'-r+c-2&\dots&\lambda_{c-1}'-r+1\\
\vdots&\vdots&\ddots&\vdots\\
\lambda_1'-r+c-e+a&\lambda_2'-r+c-e+a-1&\dots&\lambda_{c-1}'-r-e+a+2
\end{array}
\end{align*}
and by adding
\begin{align*}
\begin{array}{cccc}
&&e&\\
1&2&\dots&a-1\\
1&2&\dots&e-a
\end{array}
\end{align*}
or alternatively by replacing
\begin{align*}
\begin{array}{cccc}
\lambda_1-c+r-a&\lambda_2-c+r-a-1&\dots&\lambda_{r-1}-c+2-a
\end{array}
\end{align*}
with
\begin{align*}
\begin{array}{cccc}
\lambda_1-c+r+e-a&\lambda_2-c+r+e-a-1&\dots&\lambda_{r-1}-c+e-a+2
\end{array},
\end{align*}
by replacing
\begin{align*}
\begin{array}{cccc}
\lambda_1'-r+c-e+a-1&\lambda_2'-r+c-e+a-2&\dots&\lambda_{c-1}'-r-e+a+1
\end{array}
\end{align*}
with
\begin{align*}
\begin{array}{cccc}
\lambda_1'-r+c-1+a&\lambda_2'-r+c-2+a&\dots&\lambda_{c-1}'-r+1+a
\end{array}
\end{align*}
and by adding
\begin{align*}
\begin{array}{cccc}
&&e&\\
1&2&\dots&a-1\\
1&2&\dots&e-a
\end{array}.
\end{align*}
Now as
\begin{align*}
\frac{P_\lambda(q^2)}{P_\lambda(-q)}=\frac{(-q)^d((-q)^n+1)((-q)^{n-1}+1)\dots((-q)+1)}{\prod_h((-q)^h+1)},
\end{align*}
and as $(-q)^{t+e}\equiv(-q)^t\mod l$ for any integer $t$, then
\begin{align*}
\frac{P_\lambda(-q)}{P_\lambda(q^2)}\frac{P_\mu(q^2)}{P_\mu(-q)}&\equiv
\frac{(-q)^{\frac{a(a-1)}{2}}\prod_{i=n+1}^{n+e}((-q)^i+1)}{((-q)^e+1)\prod_{i=1}^{a-1}((-q)^i+1)\prod_{i=1}^{e-a}((-q)^i+1)}\\
&\equiv\frac{\prod_{i=n+1}^{n+e}((-q)^i+1)}{((-q)^e+1)\prod_{i=1}^{a-1}(1+(-q)^{-i})\prod_{i=1}^{e-a}((-q)^i+1)}\\
&\equiv\frac{\prod_{i=1}^e((-q)^i+1)}{\prod_{i=1}^e((-q)^i+1)}\\
&\equiv1.
\end{align*}
We now drop the assumption that the $e-$hook is straight. We prove (\ref{align:equiv}) by induction on the deviation of the hook from
being a straight hook. More precisely let $\lambda$ and $\mu$ be as in (\ref{align:equiv}) and let $\overline{\lambda}\vdash(n+1)$ be
obtained from $\lambda$ by adding a $1-$hook. Suppose this $1-$hook is added on row $r$ and column $c$ and that
$(r,c)$, $(r+1,c)$ and $(r,c+1)$ are all in the Young diagram of $\mu$. Let $\overline{\mu}\vdash(n+e+1)$ be obtained from $\mu$ by adding
a $1-$hook at $(r+1,c+1)$. We show that
\begin{align*}
\frac{P_\lambda(q^2)}{P_\lambda(-q)}\frac{P_{\overline{\lambda}}(-q)}{P_{\overline{\lambda}}(q^2)}\equiv
\frac{P_\mu(q^2)}{P_\mu(-q)}\frac{P_{\overline{\mu}}(-q)}{P_{\overline{\mu}}(q^2)}\mod l.
\end{align*}
Let's denote by $u$ and $v$ the row of the highest box and the column of the leftmost box in $\mu\backslash\lambda$ respectively.
Now
\begin{align*}
&(\sum_i\overline{\mu}_i'^2-\sum_ii\overline{\mu}_i)-(\sum_i\overline{\lambda}_i'^2-\sum_ii\overline{\lambda}_i)-
(\sum_i\mu_i'^2-\sum_ii\mu_i)+(\sum_i\lambda_i'^2-\sum_ii\lambda_i)\\
=&(\mu_c'^2+(\mu_{c+1}'+1)^2-r\mu_r-(r+1)(\mu_{r+1}+1))-(\mu_c'^2+\mu_{c+1}'^2-r\mu_r-(r+1)\mu_{r+1})\\
-&((\lambda_c'+1)^2+\lambda_{c+1}'^2-r(\lambda_r+1)-(r+1)\lambda_{r+1})+(\lambda_c'^2+\lambda_{c+1}'^2-r\lambda_r-(r+1)\lambda_{r+1})\\
=&2\mu_{c+1}'-2\lambda_c'-1\\
=&2r-2(r-1)-1\\
=&1.
\end{align*}
The hook lengths of $\overline{\lambda}$ are gotten from those of $\lambda$ by replacing
\begin{align*}
\begin{array}{cccc}
\lambda_1-c+r-1&\lambda_2-c+r-2&\dots&\lambda_{r-1}-c+1
\end{array}
\end{align*}
with
\begin{align*}
\begin{array}{cccc}
\lambda_1-c+r&\lambda_2-c+r-1&\dots&\lambda_{r-1}-c+2
\end{array},
\end{align*}
by replacing
\begin{align*}
\begin{array}{cccc}
\lambda_1'-r+c-1&\lambda_2'-r+c-2&\dots&\lambda_{c-1}'-r+1
\end{array}
\end{align*}
with
\begin{align*}
\begin{array}{cccc}
\lambda_1'-r+c&\lambda_2'-r+c-1&\dots&\lambda_{c-1}'-r+2
\end{array}
\end{align*}
and by adding $1$. Similarly the hook lengths of $\overline{\mu}$ are gotten from those of $\mu$ by replacing
\begin{align*}
\begin{array}{cccc}
\mu_1-c+r-1&\mu_2-c+r-2&\dots&\mu_r-c
\end{array}
\end{align*}
with
\begin{align*}
\begin{array}{cccc}
\mu_1-c+r&\mu_2-c+r-1&\dots&\mu_r-c+1
\end{array},
\end{align*}
by replacing
\begin{align*}
\begin{array}{cccc}
\mu_1'-r+c-1&\mu_2'-r+c-2&\dots&\mu_c'-r
\end{array}
\end{align*}
with
\begin{align*}
\begin{array}{cccc}
\mu_1'-r+c&\mu_2'-r+c-1&\dots&\mu_c'-r+1
\end{array}
\end{align*}
and by adding $1$. Note that for $i<u$, $\lambda_i=\mu_i$ and for $i<v$, $\lambda_i'=\mu_i'$. Therefore
\begin{align*}
&\frac{P_\lambda(q^2)}{P_\lambda(-q)}\frac{P_{\overline{\lambda}}(-q)}{P_{\overline{\lambda}}(q^2)}
\frac{P_\mu(-q)}{P_\mu(q^2)}\frac{P_{\overline{\mu}}(q^2)}{P_{\overline{\mu}}(-q)}\\
\equiv&(-q)\frac{((-q)^{n+e+1}+1)\prod_{i=u}^{r-1}((-q)^{\lambda_i-c+r-i+1}+1)\prod_{i=v}^{c-1}((-q)^{\lambda_i'-r+c-i+1}+1)}
{((-q)^{n+1}+1)\prod_{i=u}^{r-1}((-q)^{\lambda_i-c+r-i}+1)\prod_{i=v}^{c-1}((-q)^{\lambda_i'-r+c-i}+1)}.\\
&\frac{\prod_{i=u}^{r}((-q)^{\mu_i-c+r-i}+1)\prod_{i=v}^{c}((-q)^{\mu_i'-r+c-i}+1)}
{\prod_{i=u}^{r}((-q)^{\mu_i-c+r-i+1}+1)\prod_{i=v}^{c}((-q)^{\mu_i'-r+c-i+1}+1)}.
\end{align*}
Next we note that for all $i$ with $(u\leq i\leq r-1)$ and for all $j$ with $(v\leq j\leq c-1)$ we have
\begin{gather*}
\lambda_i=\mu_{i+1}-1\text{ and }\lambda_j'=\mu_{j+1}'-1\\
\text{and hence}\\
\lambda_i-c+r-i=\mu_{i+1}-c+r-(i+1)\text{ and }\lambda_j'-r+c-j=\mu_{j+1}'-r+c-(j+1).
\end{gather*}
Therefore we get
\begin{align*}
&(-q)\frac{((-q)^{\mu_u-c+r-u}+1)((-q)^{\mu_v'-r+c-v}+1)}{((-q)^{\mu_u-c+r-u+1}+1)((-q)^{\mu_v'-r+c-v+1}+1)}.
\end{align*}
Now by the definition of $u$ and $v$ and that the hook we add is in fact an $e-$hook we see that
\begin{align*}
\mu_u+\mu_v'-u-v+1=e.
\end{align*}
Therefore we get
\begin{align*}
&(-q)\frac{(-q)^{\mu_u-c+r-u}(1+(-q)^{-(\mu_u-c+r-u)})((-q)^{\mu_v'-r+c-v}+1)}
{(-q)^{\mu_u-c+r-u+1}(1+(-q)^{-(\mu_u-c+r-u+1)})((-q)^{\mu_v'-r+c-v+1}+1)}\\
\equiv&(-q)\frac{(-q)^{\mu_u-c+r-u}(1+(-q)^{\mu_v'-r+c-v+1})((-q)^{\mu_v'-r+c-v}+1)}
{(-q)^{\mu_u-c+r-u+1}(1+(-q)^{\mu_v'-r+c-v)})((-q)^{\mu_v'-r+c-v+1}+1)}\\
\equiv&1
\end{align*}
as desired.
\end{proof}

\begin{ack*}
The author gratefully, acknowledges financial support by ERC Advanced Grant $291512$.
\end{ack*}

\nocite{kawanaka1977, broue1990, Watanabe1999, Kessar2004, digmic1991, digne1999, shintani1976, bonnaf1998, fonsri1982}

\newpage
\bibliographystyle{plain}
\bibliography{biblio}

\end{document}